\documentclass[11pt]{article} 
\usepackage{amsmath,amssymb,amsthm}  

\allowdisplaybreaks
                                                          
\begin{document}

\def\fl#1{\left\lfloor#1\right\rfloor}
\def\cl#1{\left\lceil#1\right\rceil}
\def\ang#1{\left\langle#1\right\rangle}
\def\stf#1#2{\left[#1\atop#2\right]} 
\def\sts#1#2{\left\{#1\atop#2\right\}}
\def\eul#1#2{\left\langle#1\atop#2\right\rangle}
\def\N{\mathbb N}
\def\Z{\mathbb Z}
\def\R{\mathbb R}
\def\C{\mathbb C}

\newtheorem{theorem}{Theorem}
\newtheorem{Prop}{Proposition}
\newtheorem{Cor}{Corollary}
\newtheorem{Lem}{Lemma}
\newtheorem{Rem}{Remark}

\title{The Frobenius number for sequences of triangular numbers associated with number of solutions
}

\author{
Takao Komatsu 
\\
\small Department of Mathematical Sciences, School of Science\\[-0.8ex]
\small Zhejiang Sci-Tech University\\[-0.8ex]
\small Hangzhou 310018 China\\[-0.8ex]
\small \texttt{komatsu@zstu.edu.cn} 
}

\date{
\small MR Subject Classifications: Primary 11D07; Secondary 05A15, 05A17, 05A19, 11B68, 11D04, 11P81 
}

\maketitle
 
\begin{abstract} 
The famous linear diophantine problem of Frobenius is the problem to determine the largest integer (Frobenius number) whose number of representations in terms of $a_1,\dots,a_k$ is at most zero, that is not representable. In other words, all the integers greater than this number can be represented for at least one way. One of the natural generalizations of this problem is to find the largest integer (generalized Frobenius number) whose number of representations is at most a given nonnegative integer $p$. It is easy to find the explicit form of this number in the case of two variables. However, no explicit form has been known even in any special case of three variables. 
In this paper we are successful to show explicit forms of the generalized Frobenius numbers of the triples of triangular numbers. When $p=0$, their Frobenius number is given by Robles-P\'erez and Rosales in 2018. 
\end{abstract}

\section{Introduction}  

For integer $k\ge 2$, consider a set of positive integers $A=\{a_1,\dots,a_k\}$ with $\gcd(A)=\gcd\{a_1,\dots,a_k\}=1$.  Let ${\rm NR}(A)={\rm NR}(a_1,a_2,\dots,a_k)$ denote the set of positive integers nonrepresentable as nonnegative integer combinations of $a_1,a_2,\dots,a_k$. Then, the well-known linear Diophantine problem, posed by Sylvester \cite{sy1884}, but known as the Frobenius problem, is the problem to determine the Frobenius number $g(A)$ and the Sylvester number $n(A)$, defined by 
$$
g(A)=\max{\rm NR}(A)\quad\hbox{and}\quad n(A)=|{\rm NR}(A)|\,,   
$$ 
respectively. The Frobenius Problem has been also known as the {\it Coin Exchange Problem} (or Postage Stamp Problem / Chicken McNugget Problem), which has a long history and is one of the problems that has attracted many people as well as experts. 
For two variables $A=\{a,b\}$, it is shown that  
\begin{equation}
g(a,b)=(a-1)(b-1)-1\quad\hbox{and}\quad n(a,b)=\frac{(a-1)(b-1)}{2} 
\label{eq:g-n}
\end{equation}   
(\cite{sy1882,sy1884}). However, for three or more variables, $g(A)$ cannot be given by any set of closed formulas which can be reduced to a finite set of certain polynomials (\cite{cu90}). Only some special cases, explicit closed formulas have been found, including arithmetic, geometric-like, Fibonacci, Mersenne, and triangular (see \cite{RR18} and references therein).  
The related problem of determining the function 
$$
s(A)=\sum_{n\in{\rm NR}(A)}n\,,  
$$
which is called Sylvester sum nowadays, is introduced by Brown and Shiue \cite{bs93} and for two variables $A=\{a,b\}$, it is shown that  
\begin{equation}
s(a,b)=\frac{1}{12}(a-1)(b-1)(2 a b-a-b-1)\,. 
\label{eq:s}
\end{equation} 
Recently, the Sylvester weighted sum \cite{KZ,KZ0} is also introduced and studied.     

On the other hand, to find the number of representations $d(n;A)=d(n;a_1,a_2,\dots,a_k)$ to $a_1 x_1+a_2 x_2+\dots+a_k x_k=n$ for a given positive integer $n$ is also one of the most important and interesting topics. This number is equal to the coefficient of $x^n$ in $1/(1-x^{a_1})(1-x^{a_2})\cdots(1-x^{a_k})$ (\cite{sy1882}).   
Sylvester \cite{sy1857} and Cayley \cite{cayley} showed that $d(n;a_1,a_2,\dots,a_k)$ can be expressed as the sum of a polynomial in $n$ of degree $k-1$ and a periodic function of period $a_1 a_2\cdots a_k$. For two variables, a formula for $d(n;a_1,a_2)$ is obtained in \cite{tr00}. For three variables in the pairwise coprime case $d(n;a_1,a_2,a_3)$, in \cite{ko03}, the periodic function part is expressed in terms of trigonometric functions. However, the calculation becomes very complicated for larger $a_1,a_2,a_3$.  
In \cite{bgk01}, the explicit formula for the polynomial part is derived by using Bernoulli numbers. 

In this paper, we are interested in one of the general types of Frobenius numbers, which focuses on the number of solutions. For a given nonnegative integer $p$, consider the largest integer $g(A;p)=g(a_1,a_2,\dots,a_k;p)$ such that the number of expressions that can be represented by $a_1,a_2,\dots,a_k$ is at most $p$. That is, all integers larger than $g(A;p)$ have at least $p+1$ representations  or more. When $p=0$, $g(A)=g(A;0)$ is the original Frobenius number.  
One can consider a slightly modified number, that is, the largest integer $g'(A;p)$ such that the number of expressions is exactly $p$ (\cite{bk11}). However, for some cases any positive integer does not have exactly $p$ representations, in particular, when $p$ becomes larger. In addition, $p_1<p_2$ does not necessarily imply $g'(A;p_1)<g'(A;p_2)$. See Remark \ref{rem:4-5-6} below.   
Therefore, it would be better to treat with $g(A;p)$. Of course, after knowing $g(A;p)$ and $g(A;p-1)$, we can also get $g'(A;p)$. 
In the same stream, the generalized Sylvester number $n(A;p)=n(a_1,a_2,\dots,a_k;p)$ and generalized Sylvester sum $s(A;p)=s(a_1,a_2,\dots,a_k;p)$ are defined by 
$$
n(A;p)=\sum_{d(n;A)\le p}1\quad\hbox{and}\quad s(A;p)=\sum_{d(n;A)\le p}n\,, 
$$ 
respectively.   

For two variables, it is very easy to see that 
\begin{align}
g(a,b;p)&=(p+1)a b-a-b\,,
\label{eq:gp}\\
n(a,b;p)&=\frac{1}{2}\bigl((2 p+1)a b-a-b+1\bigr)\,,
\label{eq:np}\\
s(a,b;p)&=\frac{1}{2}a b p\bigl(a b(p+1)-a-b\bigr)\notag\\
&\quad +\frac{1}{12}(a-1)(b-1)(2 a b-a-b-1)\,.  
\label{eq:sp}   
\end{align} 
(see \cite{bb20,bk11}). When $p=0$, the classical results of $g(a,b)=g(a,b;0)$, $n(a,b)=n(a,b;0)$ in (\ref{eq:g-n}) and $s(a,b)=s(a,b;0)$ in (\ref{eq:s}) are recovered.  However, for three variables, any formula even for special triples has not been known. Recently in \cite{Ko-repunit}, some closed formulas of $g(A,p)$ and $n(A,p)$ are given when $A$ is a triple of repunits.   

In this paper, we give explicit formulas of $g(t_n,t_{n+1},t_{n+2};p)$, where $t_n=\binom{n+1}{2}$ are triangular numbers, which count the number of dots composing an equilateral triangle.  When $p=0$, the Frobenius number of three consecutive triangular numbers is give in \cite[Proposition 6]{RR18}:   
\begin{equation}
g(t_n,t_{n+1},t_{n+2})=\fl{\frac{n}{2}}(t_n+t_{n+1}+t_{n+2}-1)-1 
\label{eq:rr18}
\end{equation} 
or 
\begin{equation}
g(t_n,t_{n+1},t_{n+2})=\begin{cases}
\frac{(3 n-3)(n+1)(n+2)}{4}-1&\text{if $n$ is odd};\\ 
\frac{(3 n)(n+1)(n+2)}{4}-1&\text{if $n$ is even}. 
\end{cases}
\label{eq:rr18a}
\end{equation}

\section{Preliminaries}  

In order to prove the results, we need the following lemmas.  

\begin{Lem}[{\rm \cite[Lemma 3]{RR18}}] 
$$ 
\gcd(t_{n+1},t_{n+2})=\begin{cases}         
\dfrac{n+2}{2}&\text{if $n$ is even};\\  
n+2&\text{if $n$ is odd}\,. 
\end{cases} 
$$
\label{lem-rr18}
\end{Lem}  

\begin{Lem}[{\rm \cite[Lemma 4]{RR18}}] 
$$ 
\gcd(t_n,t_{n+1},t_{n+2})=1\,.
$$
\label{lem-rr4}
\end{Lem}

Without loss of generality, we assume that $a_1<a_2<\cdots<a_k$.   
For each $0\le i\le a_1-1$, we introduce the positive integer $m_i^{(p)}$ congruent to $i$ modulo $a_1$ such that the number of representations of $m_i^{(p)}$ is bigger than or equal to $p+1$ and that of $m_i^{(p)}-a_1$ is less than or equal to $p$. Note that $m_0^{(0)}$ is defined to be $0$.      The set 
$$
{\rm Ap}(A;p)={\rm Ap}(a_1,a_2,\dots,a_k;p)=\{m_0^{(p)},m_1^{(p)},\dots,m_{a_1-1}^{(p)}\}\,, 
$$ 
is called the {\it $p$-Ap\'ery set} of $A=\{a_1,a_2,\dots,a_k\}$ for a nonnegative integer $p$ (see, \cite{Apery}), which is congruent to the set 
$$
\{0,1,\dots,a_1-1\}\pmod{a_1}\,. 
$$

\begin{Lem}[{\rm \cite{Ko-p}}] 
For $k\ge 2$, let $A=\{a_1,a_2,\dots,a_k\}$ with $\gcd(a_1,a_2,\dots,a_k)=1$. 
For integers $p$ and $\mu$ with $p\ge 0$ and $\mu\ge 1$, we have 
\begin{align*} 
&s_\mu(A;p):=\sum_{d(n;A)\le p}n^\mu\\ 
&=\frac{1}{\mu+1}\sum_{\kappa=0}^{\mu}\binom{\mu+1}{\kappa}B_{\kappa}a_1^{\kappa-1}\sum_{i=0}^{a_1-1}\bigl(m_i^{(p)}\bigr)^{\mu+1-\kappa} 
+\frac{B_{\mu+1}}{\mu+1}(a_1^{\mu+1}-1)\,, 
\end{align*} 
where $B_n$ are Bernoulli numbers defined by 
$$
\frac{x}{e^x-1}=\sum_{n=0}^\infty B_n\frac{x^n}{n!}\,. 
$$ 
\label{th-mp}
\end{Lem}

When $\mu=0,1$ in Lemma \ref{th-mp}, together with $g(A;p)$ we have known formulae for the generalized Frobenius number, the generalized Sylvester number and the generalized Sylvester sum (see \cite{bk11}). 

\begin{Lem}  
Let $k$ and $p$ be integers with $k\ge 2$ and $p\ge 0$.  
Assume that $\gcd(A)=1$.  We have 
\begin{align}  
g(A;p)&=\max_{0\le i\le a_1-1}m_i^{(p)}-a_1
\label{mp-g}\,,\\  
n(A;p)&=\frac{1}{a_1}\sum_{i=0}^{a_1-1}m_i^{(p)}-\frac{a_1-1}{2}\,,
\label{mp-n}\\
s(A;p)&=\frac{1}{2 a_1}\sum_{i=0}^{a_1-1}\bigl(m_i^{(p)}\bigr)^2-\frac{1}{2}\sum_{i=0}^{a_1-1}m_i^{(p)}+\frac{a_1^2-1}{12}\,.
\label{mp-s}
\end{align}
\label{cor-mp}
\end{Lem}

\begin{Rem} 
When $p=0$, the identities (\ref{mp-g}), (\ref{mp-n}) and (\ref{mp-s}) are reduced to those in \cite{bs62}, \cite{se77} and \cite{tr08}, respectively.  
\end{Rem}

\section{Triangular numbers}  

When the number of variables $k\ge 3$, it is not easy to find an explicit form of Frobenius number, Sylvester number or Sylvester sum.  Nevertheless, for $p=0$, explicit forms have been discovered in some particular cases of sequences, including arithmetic, geometric-like, Fibonacci, Mersenne, and triangular (see \cite{RR18} and references therein) and so on. However, for $p\ge 1$, no explicit form has been found even in such particular cases because it is not easy to find $m_i^{(p)}$ in order to use Lemma \ref{th-mp} or Lemma \ref{cor-mp}. 

The triangular numbers $t_n=n(n+1)/2$ (\cite[A000217]{oeis}), which are elementary class of polygonal numbers described by the early Pythagoreans, represent the numbers of points used to portray equilateral triangular patterns \cite{tat05}. 
From Lemma \ref{lem-rr4}, three consecutive triangular numbers satisfy the condition $g(A)=1$. However, because of Lemma \ref{lem-rr18}, they are not relatively prime and the situation is different for even and odd numbers. Hence, more detailed discussion is needed.

\subsection{$(t_n,t_{n+1},t_{n+2})$ for general $p$}  

We can find explicit formulas for general $n$ and smaller $p$. First, we discuss some general arguments. Later, we show several explicit formulas for smaller $p$.  
 
When $p=0$, from the formula of $g_p(n):=g(t_n,t_{n+1},t_{n+2};p)$ given in (\ref{eq:rr18}) or (\ref{eq:rr18a}), we have  
$$
g_0(2)=17,~g_0(3)=29,~g_0(4)=89,~g_0(5)=125,~g_0(6)=251,g_0(7)=323,
~\dots\,. 
$$  
When $p=1$, from Theorem \ref{th:p=1}, we can have 
$$
g_1(2)=23,~g_1(3)=59,~g_1(4)=119,~g_1(5)=209,~g_1(6)=335,g_1(7)=503,
~\dots\,. 
$$  
When $p=2$, from Theorem \ref{th:p=2}, we can have  
$$
g_2(2)=29,~g_2(4)=149,~g_2(5)=230,~g_2(6)=419,g_2(7)=575,~g_2(8)=899,
~\dots\,. 
$$ 

Before mentioning the formulas for small $p$, we show the main result for general $p\ge 0$.

\begin{theorem} 
Let $p$ be a nonnegative integer.  For even $n$, 
\begin{multline*}
g\left(t_n,t_{n+1},t_{n+2};p(p+1)+\sum_{j=1}^p\cl{\frac{6 j}{n}}\right)\\ 
=\frac{(n+1)(n+2)\bigl(2(n+3)p+3 n\bigr)}{4}-1\,. 
\end{multline*}
For odd $n\ge 3$, 
\begin{multline*}
g\left(t_n,t_{n+1},t_{n+2};\sum_{j=1}^p\cl{\frac{j}{2}\left(1+\frac{3}{n}\right)}\right)\\
=\frac{(n+1)(n+2)\bigl((n+3)p+3(n-1)\bigr)}{4}-1\,. 
\end{multline*}
\label{th:012}
\end{theorem}

\begin{Rem} 
This theorem implies that any positive integer $N$ is expressed more than 
$$  
P':=\begin{cases}
p(p+1)+\sum_{j=1}^p\cl{\frac{6 j}{n}}&\text{($n$:even)}\\
\sum_{j=1}^p\cl{\frac{j}{2}\left(1+\frac{3}{n}\right)}&\text{($n$:odd)}
\end{cases}
$$  
ways in terms of three consecutive triangular numbers $t_n,t_{n+1},t_{n+2}$ if  
$$
N>N':=\begin{cases}
\frac{(n+1)(n+2)\bigl(2(n+3)p+3 n\bigr)}{4}-1&\text{($n$:even)}\\
\frac{(n+1)(n+2)\bigl((n+3)p+3(n-1)\bigr)}{4}-1&\text{($n$:odd)}. 
\end{cases}  
$$
And $N'$ itself is expressed exactly $P'$ ways. 
\end{Rem}  
\bigskip

\begin{proof}[Proof of Theorem \ref{th:012}.] 
First, let $n$ be even. Then by Lemma \ref{lem-rr18},  
$$
\ang{\dfrac{t_{n+1}}{\frac{n+2}{2}},\dfrac{t_{n+2}}{\frac{n+2}{2}}}=\ang{n+1,n+3}\,,  
$$ 
where $\ang{x_1,x_2}=\{a_1 x_1+a_2 x_2|a_1,a_2\ge 0\}$ denotes the all integers generated by the linear combination of $x_1$ and $x_2$ with nonnegative integral coefficients.    
By (\ref{eq:gp}), for $p\ge 0$, 
\begin{equation}
g(n+1,n+3;p)=(n+1)(n+3)p+(n^2+2 n-1)\,. 
\label{frobn:n+1n+3}
\end{equation} 
Thus, the length of the period of the numbers of the representation in terms of $n+1$ and $n+3$ is $(n+1)(n+3)$, and 
\begin{align*}
p&=d\bigl(g(n+1,n+3;p);n+1,n+3\bigr)\\
&=d\left(\frac{n+2}{2}g(n+1,n+3;p);t_{n+1},t_{n+2}\right)
\end{align*}  
with 
$$
d(\nu;t_{n+1},t_{n+2})=0\quad\left(\hbox{if}\quad \nu\not\equiv 0\pmod{\frac{n+2}{2}}\right)\,.
$$ 
Therefore, for a positive integer $m$,  
\begin{align}
&d\left(\frac{n+2}{2}m;t_n,t_{n+1},t_{n+2}\right)\notag\\
&=d\left(\frac{n+2}{2}m;t_{n+1},t_{n+2}\right)+d\left(\frac{n+2}{2}(m-t_n);t_{n+1},t_{n+2}\right)\notag\\
&\quad +d\left(\frac{n+2}{2}(m-2 t_n);t_{n+1},t_{n+2}\right)+\cdots\notag\\
&\quad +d\left(\frac{n+2}{2}\left(m-\fl{\frac{m}{t_n}}t_n\right);t_{n+1},t_{n+2}\right)\notag\\ 
&=d(m;n+1,n+3)+d(m-t_n;n+1,n+3)\notag\\ 
&\quad +d(m-2 t_n;n+1,n+3)+\cdots+d\left(m-t_n\fl{\frac{m}{t_n}};n+1,n+3\right)\,.
\label{eq:m-even-n} 
\end{align} 
Hence, when $m=g(n+1,n+3;p)=(n+1)(n+3)p+(n^2+2 n-1)$, we have 
\begin{align*}
&d\left(\frac{n+2}{2}g(n+1,n+3;p);t_n,t_{n+1},t_{n+2}\right)\\ 
&=\sum_{j=0}^{\fl{g(n+1,n+3;p)/t_n}}d\bigl(g(n+1,n+3;p)-j t_n;n+1,n+3\bigr)\,. 
\end{align*} 
Here, each value of $d\bigl(g(n+1,n+3;p)-j t_n;n+1,n+3\bigr)$ must be equal to any of $0,1,\cdots,p$. For convenience, put 
$$
g_j'=g_{j,p}':=g(n+1,n+3;p)-j t_n\,. 
$$ 
Remember that the largest integer whose number of representations in terms of $n+1$ and $n+3$ is equal to $p$ is given in (\ref{frobn:n+1n+3}). For example, assume that we want to know the values $j$ such that $d((n+1)(n+3)p+(n^2+2 n-1)-j t_n;n+1,n+3)=p-2$. This has been achieved by solving $g_{p-3}<g_j'\le g_{p-2}$. By 
$$
-3(n+1)(n+3)<-j t_n\le -2(n+1)(n+3)\,, 
$$   
we get 
$$
6+\frac{18}{n}>j\ge 4+\frac{12}{n}\,. 
$$   
Hence, 
$$
j=4+\cl{\frac{12}{n}},5+\cl{\frac{12}{n}},\dots,5+\cl{\frac{18}{n}}\,. 
$$ 
It implies that the number of values $(p-2)$ is given by  
$$
2+\cl{\frac{18}{n}}-\cl{\frac{12}{n}}
$$  
among the sequence $d\bigl(g(n+1,n+3;p)-j t_n;n+1,n+3\bigr)$ for $0\le j\le g(n+1,n+3;p)/t_n$. 
Similarly, the numbers of values $p$, $(p-1)$, $\dots$, $1$ are given by 
\begin{align*}
&2+\cl{\frac{6}{n}}\,,\\
&2+\cl{\frac{12}{n}}-\cl{\frac{6}{n}}\,,\\
&\qquad \dots\,,\\
&2+\cl{\frac{6 p}{n}}-\cl{\frac{6(p-1)}{n}}\,, 
\end{align*}
respectively.  
Therefore, 
\begin{align*}
&d\left(\frac{n+2}{2}g(n+1,n+3;p)+\left(\frac{n+2}{2}-1\right)t_n;t_n,t_{n+1},t_{n+2}\right)\\
&=d\left(\frac{(n+1)(n+2)\bigl(2(n+3)p+3 n\bigr)}{4}-1;t_n,t_{n+1},t_{n+2}\right)\\ 
&=p\left(2+\cl{\frac{6}{n}}\right)+(p-1)\left(2+\cl{\frac{12}{n}}-\cl{\frac{6}{n}}\right)\\
&\quad+(p-2)\left(2+\cl{\frac{18}{n}}-\cl{\frac{12}{n}}\right)+\cdots+1\cdot\left(2+\cl{\frac{6 p}{n}}-\cl{\frac{6(p-1)}{n}}\right)\\
&=p(p+1)+\sum_{j=1}^p\cl{\frac{6 j}{n}}\,. 
\end{align*} 

Next, let $n$ be odd. In the odd case the reasoning is similar to that of the even case. However, we have preferred to include all computations because of differences in some details. By Lemma \ref{lem-rr18},  
$$
\ang{\dfrac{t_{n+1}}{n+2},\dfrac{t_{n+2}}{n+2}}=\ang{\frac{n+1}{2},\frac{n+3}{2}}\,. 
$$ 
By (\ref{eq:gp}), for $p\ge 0$, 
\begin{equation}
g\left(\frac{n+1}{2},\frac{n+3}{2};p\right)=\frac{(n+1)(n+3)}{4}p+\frac{n^2-5}{4}\,. 
\label{frobn:n+1n+3odd}
\end{equation} 
Thus, the length of the period of the numbers of the representation in terms of $n+1$ and $n+3$ is $(n+1)(n+3)/4$, and 
\begin{align*}
p&=d\left(g\left(\frac{n+1}{2},\frac{n+3}{2};p\right);\frac{n+1}{2},\frac{n+3}{2}\right)\\
&=d\left((n+2)g\left(\frac{n+1}{2},\frac{n+3}{2};p\right);t_{n+1},t_{n+2}\right)
\end{align*}  
with 
$$
d(\nu;t_{n+1},t_{n+2})=0\quad\left(\hbox{if}\quad \nu\not\equiv 0\pmod{n+2}\right)\,.
$$ 
Therefore, for a positive integer $m$,  
\begin{align}
&d\left((n+2)m;t_n,t_{n+1},t_{n+2}\right)\notag\\
&=d\left((n+2)m;t_{n+1},t_{n+2}\right)+d\left((n+2)(m-t_n);t_{n+1},t_{n+2}\right)\notag\\
&\quad +d\left((n+2)(m-2 t_n);t_{n+1},t_{n+2}\right)+\cdots\notag\\
&\quad +d\left((n+2)\left(m-\fl{\frac{m}{t_n}}t_n\right);t_{n+1},t_{n+2}\right)\notag\\ 
&=d\left(m;\frac{n+1}{2},\frac{n+3}{2}\right)+d\left(m-t_n;\frac{n+1}{2},\frac{n+3}{2}\right)\notag\\ 
&\quad +d\left(m-2 t_n;\frac{n+1}{2},\frac{n+3}{2}\right)+\cdots+d\left(m-t_n\fl{\frac{m}{t_n}};\frac{n+1}{2},\frac{n+3}{2}\right)\,. 
\label{eq:m-odd-n}
\end{align} 
Hence, when $m=g\bigl((n+1)/2,(n+3)/2;p\bigr)=(n+1)(n+3)p/4+(n^2-5)/4$, we have 
\begin{align*}
&d\left((n+2)g\left(\frac{n+1}{2},\frac{n+3}{2};p\right);t_n,t_{n+1},t_{n+2}\right)\\
&=\sum_{j=0}^{\fl{g(\frac{n+1}{2},\frac{n+3}{2};p)/t_n}}d\left(g\left(\frac{n+1}{2},\frac{n+3}{2};p\right)-j t_n;\frac{n+1}{2},\frac{n+3}{2}\right)\,. 
\end{align*} 
Here, each value of $d\bigl(g((n+1)/2,(n+3)/2;p)-j t_n;(n+1)/2,(n+3)/2\bigr)$ must be equal to any of $0,1,\cdots,p$. For convenience, put 
$$
g_j'=g_{j,p}':=g\left(\frac{n+1}{2},\frac{n+3}{2};p\right)-j t_n\,. 
$$ 
Remember that the largest integer whose number of representations in terms of $(n+1)/2$ and $(n+3)/2$ is equal to $p$ is given in (\ref{frobn:n+1n+3}). For example, assume that we want to know the values $j$ such that $d\bigl((n+1)(n+3)p+(n^2+2 n-1)-j t_n;(n+1)/2,(n+3)/2\bigr)=p-2$. This has been achieved by solving $g_{p-3}<g_j'\le g_{p-2}$. By 
$$
-\frac{3(n+1)(n+3)}{4}<-j t_n\le -\frac{2(n+1)(n+3)}{4}\,, 
$$   
we get 
$$
\frac{3}{2}\left(1+\frac{3}{n}\right)>j\ge \frac{2}{2}\left(1+\frac{3}{n}\right)\,. 
$$   
Hence, the number of values $(p-2)$ is given by  
$$
\cl{\frac{3}{2}\left(1+\frac{3}{n}\right)}-\cl{\frac{2}{2}\left(1+\frac{3}{n}\right)}
$$  
among the sequence $d\bigl(g((n+1)/2,(n+3)/2;p)-j t_n;(n+1)/2,(n+3)/2\bigr)$ for $0\le j\le g((n+1)/2,(n+3)/2;p)/t_n$. 
Similarly, the numbers of values $p$, $(p-1)$, $\dots$, $1$ are given by 
\begin{align*}
&\cl{\frac{1}{2}\left(1+\frac{3}{n}\right)}\,,\\
&\cl{\frac{2}{2}\left(1+\frac{3}{n}\right)}-\cl{\frac{1}{2}\left(1+\frac{3}{n}\right)}\,,\\
&\qquad \dots\,,\\
&\cl{\frac{p}{2}\left(1+\frac{3}{n}\right)}-\cl{\frac{p-1}{2}\left(1+\frac{3}{n}\right)}\,, 
\end{align*}
respectively.  
Therefore, 
\begin{align*}
&d\left((n+2)g\left(\frac{n+1}{2},\frac{n+3}{2};p\right)+(n+1)t_n;t_n,t_{n+1},t_{n+2}\right)\\
&=d\left(\frac{(n+1)(n+2)\bigl((n+3)p+3(n-1)\bigr)}{4}-1;t_n,t_{n+1},t_{n+2}\right)\\ 
&=\cl{\frac{1}{2}\left(1+\frac{3}{n}\right)}+(p-1)\left(\cl{\frac{2}{2}\left(1+\frac{3}{n}\right)}-\cl{\frac{1}{2}\left(1+\frac{3}{n}\right)}\right)\\
&\quad+(p-2)\left(\cl{\frac{3}{2}\left(1+\frac{3}{n}\right)}-\cl{\frac{2}{2}\left(1+\frac{3}{n}\right)}\right)+\cdots\\
&\quad+1\cdot\left(\cl{\frac{p}{2}\left(1+\frac{3}{n}\right)}-\cl{\frac{p-1}{2}\left(1+\frac{3}{n}\right)}\right)\\
&=\sum_{j=1}^p\cl{\frac{j}{2}\left(1+\frac{3}{n}\right)}\,. 
\end{align*}  
\end{proof}

\subsection{$(t_n,t_{n+1},t_{n+2})$ for $p=1$}  

We consider the general $n$ case with the number of representations is at most $1$ to give an explicit formula.  
 
\begin{theorem}  
$$
g(t_n,t_{n+1},t_{n+2};1)=n(n+1)(n+2)-1\quad(n\ge 2)\,. 
$$ 
\label{th:p=1}  
\end{theorem}  
\begin{proof} 
When $n$ is even, we can choose 
$$
m=\frac{(3 n-1)(n+2)}{2}=g(n+1,n+3;0)+t_n\,. 
$$
in (\ref{eq:m-even-n}).  
Since the only nonnegative integral solution of 
$$
m=a(n+1)+b(n+3)\quad(n\hbox{:even})  
$$ 
is given by $a=(n-2)/2$ and $b=n$, together with the definition of $g(n+1,n+3;0)$, 
we see that 
$$
d(m;n+1,n+3)=1\quad\hbox{and}\quad d(m-j t_n;n+1,n+3)=0\quad(j\ge 1)\,. 
$$  
Since $g(n+1,n+3;1)-t_n=2 n^2+6 n+2-n(n+1)/2>g(n+1,n+3;0)$, the number of representations of $g(n+1,n+3;1)-t_n$ cannot be $0$. So, $m=g(n+1,n+3;1)$ is impossible.   
Hence, by (\ref{eq:m-even-n}), we have 
\begin{align*}
g(t_n,t_{n+1},t_{n+2};1)&=
\frac{n+2}{2}m+\left(\frac{n-2}{2}+1\right)t_n\\
&=n^3+3 n^2+2 n-1\,. 
\end{align*} 
When $n$ is odd with $n\ge 3$, we can choose 
$$
m=\frac{n^2+2 n-1}{2}=g\left(\frac{n+1}{2},\frac{n+3}{2};1\right)\,. 
$$
in (\ref{eq:m-odd-n}). Now, 
$$
m-t_n=\frac{n-1}{2}
$$  
cannot be represented in terms of $(n+1)/2$ and $(n+3)/2$. 
Hence, by (\ref{eq:m-odd-n}), we have 
\begin{align*}
g(t_n,t_{n+1},t_{n+2};1)&=(n+2)m+\left(n+2-1\right)t_n\\
&=n^3+3 n^2+2 n-1\,. 
\end{align*}  
\end{proof}

\subsection{$(t_n,t_{n+1},t_{n+2})$ for $p=2$}  

When $p=2$, we have the following explicit form.  
 
\begin{theorem}  
$$
g(t_n,t_{n+1},t_{n+2};2)=\begin{cases}
\frac{(5 n-3)(n+1)(n+2)}{4}-1&\text{if $n\ge 5$ is odd};\\
\frac{(5 n)(n+1)(n+2)}{4}-1&\text{if $n\ge 2$ is even}\,. 
\end{cases}
$$ 
\label{th:p=2}  
\end{theorem}  

\begin{Rem} 
When $n=3$, no positive integer can be represented in exactly two ways in terms of $(t_3,t_4,t_5)$. In this case, we can include the case where $n=3$ as $g(t_3,t_4,t_5;2)=g(t_3,t_4,t_5;1)=59$. 
\end{Rem} 

\begin{proof}[Proof of Theorem \ref{th:p=2}.] 
Let $n$ be even. The number $p=2$ may be achieved in (\ref{eq:m-even-n}) as $d(m;n+1,n+3)=2$ and $d(m-j t_n;n+1,n+3)=0$ ($j\ge 1$).  
Then the number $p=2$ may be achieved in (\ref{eq:m-even-n}) as $d(m;n+1,n+3)=d(m-t_n;n+1,n+3)=1$ and $d(m-j t_n;n+1,n+3)=0$ ($j\ge 2$). There are two possibilities. But if $m=g(n+1,n+3;1)=2 n^2+6 n+2$, by $g(n+1,n+3;1)-2 t_n=n^2+5 n+2>g(n+1,n+3;0)=n^2+2 n-1$, this case is invalid. If $m=2 n^2+3 n-1=g(n+1,n+3;0)+2 t_n$, 
$$ 
m=(n-1)(n+1)+n(n+3)\quad\hbox{and}\quad m-t_n=\frac{n-2}{2}(n+1)+n(n+3) 
$$ 
are the only $1$ expressions, respectively.  Therefore, we have 
\begin{align*}
g(t_n,t_{n+1},t_{n+2};2)&=\frac{n+2}{2}m+\frac{n}{2}t_n\\
&=\frac{5 n^3+15 n^2+10 n-4}{4}\,. 
\end{align*}
Let $n$ be odd. We can only choose 
$$
m=\frac{(3 n+5)(n-1)}{4}=g\left(\frac{n+1}{2},\frac{n+3}{2};0\right)+t_n\,. 
$$
in (\ref{eq:m-odd-n}). In this case, 
the number $p=2$ is achieved in (\ref{eq:m-even-n}) as $d\bigl(m;(n+1)/2,(n+3)/2\bigr)=2$ and $d\bigl(m-j t_n;n+1,n+3\bigr)=0$ ($j\ge 1$).  In fact, $m$ can be expressed in two ways as 
$$
m=(n-1)\frac{n+1}{2}+\frac{n+1}{2}\frac{n+3}{2}=\frac{n-5}{2}\frac{n+1}{2}+n\frac{n+3}{2}\quad(n\ge 5)
$$ 
and $g(\frac{n+1}{2},\frac{n+3}{2};0)=m-t_n=(n^2-5)/4$ for $j=1$. 
Therefore, we have 
\begin{align*}
g(t_n,t_{n+1},t_{n+2};1)&=(n+2)m+\left(n+2-1\right)t_n\\
&=\frac{5 n^3+12 n^2+n-10}{4}\,. 
\end{align*}  
\end{proof}

\subsection{$(t_n,t_{n+1},t_{n+2})$ for $p\ge 3$}  

When $3\le p\le 10$, we have the following explicit forms.  We omit the details. When $p$ is larger, the arguments become more complicated.  If n is small in the following formulas, they can be calculated individually, so the details are also omitted.  
 
\begin{theorem}  
$$
g(t_n,t_{n+1},t_{n+2};3)=\begin{cases}
\frac{(5 n+3)(n+1)(n+2)}{4}-1&\text{if $n\ge 3$ is odd};\\
\frac{(5 n+6)(n+1)(n+2)}{4}-1&\text{if $n\ge 6$ is even}\,.  
\end{cases}
$$ 
For $n=5$ and $n\ge 7$, we have  
$$
g(t_n,t_{n+1},t_{n+2};4)=\frac{3 n(n+1)(n+2)}{2}-1\,.
$$ 
For $n=6$ and $n\ge 8$, we have  
$$
g(t_n,t_{n+1},t_{n+2};5)=\frac{3(n+1)^2(n+2)}{2}-1\,. 
$$ 
We have 
\begin{align*}
g(t_n,t_{n+1},t_{n+2};6)&=\begin{cases}
\frac{(7 n-3)(n+1)(n+2)}{4}-1&\text{if $n\ge 11$ is odd};\\
\frac{7 n(n+1)(n+2)}{4}-1&\text{if $n\ge 8$ is even}\,, 
\end{cases}\\ 
g(t_n,t_{n+1},t_{n+2};7)&=\begin{cases}
\frac{(7 n+3)(n+1)(n+2)}{4}-1&\text{if $n\ge 5$ is odd};\\
\frac{(7 n+6)(n+1)(n+2)}{4}-1&\text{if $n\ge 6$ is even}\,, 
\end{cases}\\ 
g(t_n,t_{n+1},t_{n+2};8)&=\begin{cases}
\frac{(7 n+9)(n+1)(n+2)}{4}-1&\text{if $n\ge 9$ is odd};\\
\frac{(7 n+12)(n+1)(n+2)}{4}-1&\text{if $n\ge 12$ is even}\,. 
\end{cases}
\end{align*}  
For $n=11$ and $n\ge 13$, we have  
$$
g(t_n,t_{n+1},t_{n+2};9)=2 n(n+1)(n+2)-1\,. 
$$ 
For $n=3,4$ and $n\ge 8$, we have  
$$
g(t_n,t_{n+1},t_{n+2};10)=\frac{(4 n+3)(n+1)(n+2)}{2}-1\,. 
$$ 
\label{th:p=3-10}  
\end{theorem}

\noindent 
{\bf Conjecture.}
For some nonnegative integer $p$, there exists an odd integer $q$ and integers $n_j$ ($j=1,2,3,4$) such that 
$$
g(t_n,t_{n+1},t_{n+2};p)=\begin{cases}
\frac{(q n-3)(n+1)(n+2)}{4}&\text{if $n\ge n_1$ is odd};\\
\frac{(q n)(n+1)(n+2)}{4}&\text{if $n\ge n_2$ is even} 
\end{cases}
$$
and   
$$
g(t_n,t_{n+1},t_{n+2};p+1)=\begin{cases}
\frac{(q n+3)(n+1)(n+2)}{4}&\text{if $n\ge n_3$ is odd};\\
\frac{(q n+6)(n+1)(n+2)}{4}&\text{if $n\ge n_4$ is even}
\end{cases}
$$
and so on.  
For some nonnegative integer $p'$, there exists an even integer $q'$ and integers $n_5$ and $n_6$ such that 
$$
g(t_n,t_{n+1},t_{n+2};p')=\frac{(q' n)(n+1)(n+2)}{4}\quad(n\ge n_5)
$$ 
and 
$$
g(t_n,t_{n+1},t_{n+2};p'+1)=\frac{(q' n+6)(n+1)(n+2)}{4}\quad(n\ge n_6)
$$ 
and so on.

\section{Triangular triples for smaller $n$} 

On the contrary, if $n$ is smaller, we can find some explicit formulas for larger $p$.  In this section, we consider the cases for $n=2,3,4$.

\subsection{$(t_3,t_4,t_5)$}  

When $p$ is also a triangular number, we can obtain explicit formulas for general positive integer $p$. 

\begin{Prop}
For $p\ge 1$, 
\begin{align*}
g(t_3,t_4,t_5;t_{p-1})&=30 p-1\,,\\ 
n(t_3,t_4,t_5;t_{p-1})&=30 p-16\,,\\
s(t_3,t_4,t_5;t_{p-1})&=15(30 p^2-31 p+12)\,.  
\end{align*}   
\label{prp:tri345} 
\end{Prop}

In fact, there does not exist a positive integer $n$ such that for some nonnegative integer $m$ $d(n;6,10,15)\ne t_m$. 

\begin{proof}[Sketch of the proof of Proposition \ref{prp:tri345}.]
Since $(t_3,t_4,t_5)=(6,10,15)$, we consider the representation $10 x_2+15 x_3=5(2 x_2+3 x_3)$ ($x_2,x_3\ge 0$).  
As the generating function  
\begin{align*}
&\frac{1}{(1-z^2)(1-z^3)}\\
&=\frac{1}{2(1-z^2)}+\frac{1}{3(1-z^3)}-\frac{z}{3(1-z^3)}+\frac{1}{6(1-z)^2}\\
&=\frac{1}{2}\sum s_1(n)z^n+\frac{1}{3}\sum s_2(n)z^n-\frac{1}{3}\sum s_3(n)z^n+\frac{1}{6}\sum(n+1)z^n\,, 
\end{align*}
where 
\begin{align*}
&s_1(n)=\begin{cases}
1&\text{$n\equiv 0\pmod 2$};\\
0&\text{otherwise},
\end{cases}\qquad 
s_2(n)=\begin{cases}
1&\text{$n\equiv 0\pmod 3$};\\
0&\text{otherwise},
\end{cases}\\  
&s_3(n)=\begin{cases}
1&\text{$n\equiv 1\pmod 3$};\\
0&\text{otherwise}\,,
\end{cases}
\end{align*} 
the sequence is given by 
\begin{align}  
&1, 0, 1, 1, 1, 1, 2, 1, 2, 2, 2, 2, 3, 2, 3, 3, 3, 3, 4, 3, 4, 4, 4, 4, 5, 4, 5, 5, 5, 5, 6, 5, \dots\notag\\
&=1,0,\overline{m,m,m,m,m+1,m}_{m=1}^\infty\label{seq6}\\ 
&=\left\{\fl{\frac{n+2}{2}}-\fl{\frac{n+2}{3}}\right\}_{n\ge 0}\notag 
\end{align}
({\it Cf.\/} \cite[A008615]{oeis}).  
Hence, for $n\ge 0$ 
$$
d(n;2,3)=\fl{\frac{n+2}{2}}-\fl{\frac{n+2}{3}}= 
d(5 n;10,15)\,,
$$ 
and for $p\ge 0$, $g(2,3;p)=6 p+1$.  Thus, 
\begin{align*} 
m_5^{(p')}\pmod 6=\max_{0\le i\le 5}m_i^{(p')}=5(6 p+1)\,. 
\end{align*}
Here, because the number of representations of $30 p+5$ in terms of $(6,10,15)$ is just $p$ more than that of $30p-1$, by applying (\ref{mp-g}) of Lemma \ref{cor-mp}, we have 
$$
g\bigl(t_3,t_4,t_5;p'\bigr)=30 p-1\,, 
$$ 
where $p'=1+2+\cdots+(p-1)=t_{p-1}$.  
Similarly, from the sequence (\ref{seq6}), we see that for $p'=t_{p-1}$ 
\begin{align*}  
m_1^{(p')}&=5(6 p-1),\quad m_2^{(p')}=5(6 p-2),\quad m_3^{(p')}=5(6 p-3),\\
m_4^{(p')}&=5(6 p-4),\quad m_0^{(p')}=5(6 p-6)\,. 
\end{align*}
By applying (\ref{mp-n}) and (\ref{mp-s}) of Lemma \ref{cor-mp}, we have the second and the third identities.   
\end{proof}

For example, since 
$$
\{(x_2,x_3)\mid 2 x_2+3 x_3=19,\, x_2,x_3\ge 0\}=\{(2,5),(5,3),(8,1)\}\,,
$$ 
we see that 
\begin{align*} 
&\{(x_1,x_2,x_3)\mid 6 x_1+5(2 x_2+3 x_3)=89,\, x_1,x_2,x_3\ge 0\}\\
&=\{(4,2,3),(4,5,1),(9,2,1)\}\,, \\
&\{(x_1,x_2,x_3)\mid 6 x_1+5(2 x_2+3 x_3)=95,\, x_1,x_2,x_3\ge 0\}\\
&=\{(5,2,3),(5,5,1),(10,2,1),(0,2,5),(0,5,3),(0,8,1)\}\,. 
\end{align*} 
\bigskip

\subsection{$(t_2,t_3,t_4)$}  

Next, we consider the case $n=2$.  

\begin{Prop}
For $n\ge 0$ and $j=1,2,3,4,5$, 
\begin{align*}
g\bigl(t_2,t_3,t_4;(n+1)(5 n+2 j)/2\bigr)&=30 n+6j+17\,,\\ 
n\bigl(t_2,t_3,t_4;(n+1)(5 n+2 j)/2\bigr)&=30 n+6j+8\,,\\
s\bigl(t_2,t_3,t_4;(n+1)(5 n+2 j)/2\bigr)&=3\bigl(150 n^2+5(12 j+17)n+6 j^2+17 j+23\bigr)\,.  
\end{align*}   
\label{prp:tri234} 
\end{Prop} 
\begin{proof}  
By applying Lemma \ref{cor-mp} for Lemma \ref{lem:tri234}, we get the results.  
\end{proof}

\begin{Lem}  
For $n\ge 0$ and $j=1,2,3,4,5$, 
\begin{align*} 
m_0^{((n+1)(5 n+2 j)/2)}&=30 n+6 j\,,\\ 
m_1^{((n+1)(5 n+2 j)/2)}&=30 n+6 j+10\,,\\ 
m_2^{((n+1)(5 n+2 j)/2)}&=30 n+6 j+20\,. 
\end{align*} 
\label{lem:tri234} 
\end{Lem}  
\begin{proof}
By the generating function 
$$
\frac{1}{(1-z^3)(1-z^5)}=\sum_{n=0}^\infty d(n)z^n\,, 
$$   
for $n\ge 0$,  
the sequence of number of representations $d(n)$ of $n=3 x_2+5 x_3$ ($x_2,x_3\ge 0$) can be given as 
$$
(1,\overline{m-1,m-1,m,m-1,m,m,m-1,\underbrace{m,m,m,m,m,m,m}_{7},m+1})_{m=1}^\infty\,. 
$$ 
Hence, for two variables $(3,5)$, we see that 
$$
g(3,5;p)=15 p+7\quad (p\ge 0)\,. 
$$ 
For two variables $(6,10)$, the same representation appears at $2(15 p+7)$. In this case, odd numbers cannot be represented in terms of $6$ and $10$. Hence, 
\begin{align*} 
&d(2 n;3,6,10)=d(2 n+3;3,6,10)\\ 
&=d(2 n;6,10)+d(2 n-6;6,10)+d(2 n-12;6,10)+\cdots+
\begin{cases}
d(4;6,10)&\\
d(2;6,10)&\\
d(0;6,10)&
\end{cases}\\
&=d(n;3,5)+d(n-3;3,5)+d(n-6;3,5)+\cdots+
\begin{cases}
d(2;3,5)&\\
d(1;3,5)&\\
d(0;3,5)\,.&
\end{cases}
\end{align*}
Thus, for $p\ge 0$, the number of representations of $30 p+17$ (or of $30 p+14$) in terms of $(3,6,10)$ is given by 
\begin{align*}  
&d(15 p+7;3,5)+d(15 p+4;3,5)+d(15 p+1;3,5)+\\
&\qquad \cdots+d(4;3,5)+d(1;3,5)\\
&=\underbrace{p+\cdots+p}_5+\underbrace{(p-1)+\cdots+(p-1)}_5+\cdots+\underbrace{1+\cdots+1}_5+0+0+0\\
&=\frac{5 p(p+1)}{2}\,. 
\end{align*} 
Similarly, for $p\ge 1$, the numbers of representations of $30 p+11$ (or of $30 p+8$), $30 p+5$ (or of $30 p+2$), $30 p-1$ (or of $30 p-4$) and $30 p-7$ (or of $30 p-10$) in terms of $(3,6,10)$ are given by 
\begin{align}
&\frac{5 p(p+1)}{2}-p=\frac{p(5 p+3)}{2}\,,&\quad 
&\frac{5 p(p+1)}{2}-2 p=\frac{p(5 p+1)}{2}\,,\notag\\ 
&\frac{5 p(p+1)}{2}-3 p=\frac{p(5 p-1)}{2}\,,&\quad 
&\frac{5 p(p+1)}{2}-4 p=\frac{p(5 p-3)}{2}\,, 
\label{eq:5pp}
\end{align} 
respectively.  These are all the cases for represented integers equivalent to $2\pmod 3$.  
Next, consider the cases for represented integers equivalent to $0\pmod 3$. For $p\ge 1$, the number of representations of $30 p-3$ (or of $30 p-6$) in terms of $(3,6,10)$ is given by 
\begin{align*}  
&d(15 p-3;3,5)+d(15 p-6;3,5)+d(15 p-9;3,5)\\
&\qquad+\cdots+d(6;3,5)+d(3;3,5)+d(0;3,5)\\
&=\underbrace{p+\cdots+p}_5+\underbrace{(p-1)+\cdots+(p-1)}_5+\cdots+\underbrace{2+\cdots+2}_5+1+1+1+1+1\\
&=\frac{5 p(p+1)}{2}\,.  
\end{align*} 
Similarly, for $p\ge 1$, the numbers of representations of $30 p-9$ (or of $30 p-12$), $30 p-15$ (or of $30 p-18$), $30 p-21$ (or of $30 p-24$) and $30 p-27$ (or of $30 p-30$) in terms of $(3,6,10)$ are given as in (\ref{eq:5pp}), respectively.  
Finally, consider the cases for represented integers equivalent to $1\pmod 3$. For $p\ge 1$, the number of representations of $30 p+7$ (or of $30 p+4$) in terms of $(3,6,10)$ is given by 
\begin{align*}  
&d(15 p+2;3,5)+d(15 p-1;3,5)+d(15 p-4;3,5)\\
&\qquad+\cdots+d(5;3,5)+d(2;3,5)\\
&=\underbrace{p+\cdots+p}_5+\underbrace{(p-1)+\cdots+(p-1)}_5+\cdots+\underbrace{1+\cdots+1}_5+0\\
&=\frac{5 p(p+1)}{2}\,. 
\end{align*} 
Similarly, for $p\ge 1$, the numbers of representations of $30 p+1$ (or of $30 p-2$), $30 p-5$ (or of $30 p-8$), $30 p-11$ (or of $30 p-14$) and $30 p-17$ (or of $30 p-20$) in terms of $(3,6,10)$ are given as in (\ref{eq:5pp}), respectively. 
Therefore, for $p\ge 1$ and $i=0,1,2,3,4$, we have $m_2^{(5 p(p+1)/2-i p)}=30 p+20-6 i$, $m_1^{(5 p(p+1)/2-i p)}=30 p+10-6 i$, $m_0^{(5 p(p+1)/2-i p)}=30 p-6 i$ and 
$$
g(3,6,10;5 p(p+1)/2-i p)=30 p+17-6 i\,. 
$$ 
\end{proof}

\subsection{$(t_4,t_5,t_6)$} 

Consider the case $(10,15,21)$ when $n=4$. For simplicity, for nonnegative integers $n$ and $j$, put 
\begin{align*}
r_e=r_e(n,j)&=\frac{n(7 n+8-2 j)}{4}\,,\\ 
r_o=r_o(n,j)&=\frac{7 n^2+(8-4 j)n+1}{4},,\\  
r_{oo}=r_{oo}(n,j)&=\frac{7 n^2+(2-4 j)n-1}{4}\,. 
\end{align*} 

\begin{Prop}  
For even $n$, when $n\ge 0$ and $j=0$, or when $n\ge 2$ and $j=2,3,4,5,6,8$, we have  
\begin{align*} 
g\bigl(t_4,t_5,t_6;r_e\bigr)&=105 n+89-15 j\,,\\  
n\bigl(t_4,t_5,t_6;r_e\bigr)&=105 n+44-15 j\,,\\  
s\bigl(t_4,t_5,t_6;r_e\bigr)&=\frac{15}{2}\bigl(735 n^2-7(60 j-89)n+60 j^2-178 j+194\bigr)\,. 
\end{align*}  
For odd $n$, when $n\ge 1$ and $j=0$, or when $n\ge 3$ and $j=1,2,3,4$, we have  
\begin{align*} 
g\bigl(t_4,t_5,t_6;r_o\bigr)&=105 n+89-30 j\,,\\
n\bigl(t_4,t_5,t_6;r_o\bigr)&=105 n+44-30 j\,,\\
s\bigl(t_4,t_5,t_6;r_o\bigr)&=\frac{15}{2}\bigl(735 n^2-7(30 j-89)n+60 j^2-89 j+194\bigr)\,.  
\end{align*}  
For odd $n\ge 1$ and $j=0,1$, we have 
\begin{align*} 
g\bigl(t_4,t_5,t_6;r_{oo}\bigr)&=105 n+44-30 j\,,\\
n\bigl(t_4,t_5,t_6;r_{oo}\bigr)&=105 n-1-30 j\,,\\
s\bigl(t_4,t_5,t_6;r_{oo}\bigr)&=\frac{15}{2}\bigl(735 n^2-7(60 j+1)n+60 j^2+2 j+62\bigr)\,.  
\end{align*}  
\label{prp:tri456}  
\end{Prop}

The results in Proposition \ref{prp:tri456} immediately follows from the next lemma.

\begin{Lem}  
For even $n$, when $n\ge 0$ and $j=0$, or when $n\ge 2$ and $j=2,4,6,8$, we have  
\begin{align*}  
&m_9^{(r_e)}=105 n+99-15 j,\quad m_4^{(r_e)}=105 n+84-15 j,\\ 
&m_8^{(r_e)}=105 n+78-15 j,\quad m_3^{(r_e)}=105 n+63-15 j,\\ 
&m_7^{(r_e)}=105 n+57-15 j,\quad m_2^{(r_e)}=105 n+42-15 j,\\ 
&m_6^{(r_e)}=105 n+36-15 j,\quad m_1^{(r_e)}=105 n+21-15 j,\\ 
&m_5^{(r_e)}=105 n+15-15 j,\quad m_0^{(r_e)}=105 n-15 j\,. 
\end{align*}
For even $n$, when $n\ge 2$ and $j=3,5$, we have  
\begin{align*}  
&m_4^{(r_e)}=105 n+99-15 j,\quad m_9^{(r_e)}=105 n+84-15 j,\\ 
&m_3^{(r_e)}=105 n+78-15 j,\quad m_8^{(r_e)}=105 n+63-15 j,\\ 
&m_2^{(r_e)}=105 n+57-15 j,\quad m_7^{(r_e)}=105 n+42-15 j,\\ 
&m_1^{(r_e)}=105 n+36-15 j,\quad m_6^{(r_e)}=105 n+21-15 j,\\ 
&m_0^{(r_e)}=105 n+15-15 j,\quad m_5^{(r_e)}=105 n-15 j\,. 
\end{align*}
For odd $n$, when $n\ge 1$ and $j=0$, or when $n\ge 3$ and $j=1,2,3,4$, we have  
\begin{align*}  
&m_4^{(r_o)}=105 n+99-30 j,\quad m_9^{(r_o)}=105 n+84-30 j,\\ 
&m_3^{(r_o)}=105 n+78-30 j,\quad m_8^{(r_o)}=105 n+63-30 j,\\ 
&m_2^{(r_o)}=105 n+57-30 j,\quad m_7^{(r_o)}=105 n+42-30 j,\\ 
&m_1^{(r_o)}=105 n+36-30 j,\quad m_6^{(r_o)}=105 n+21-30 j,\\ 
&m_0^{(r_o)}=105 n+15-30 j,\quad m_5^{(r_o)}=105 n-30 j\,.
\end{align*}  
For odd $n\ge 1$ and $j=0,1$, we have 
\begin{align*}  
&m_9^{(r_{oo})}=105 n+54-30 j,\quad m_4^{(r_{oo})}=105 n+39-30 j,\\ 
&m_8^{(r_{oo})}=105 n+33-30 j,\quad m_3^{(r_{oo})}=105 n+18-30 j,\\ 
&m_7^{(r_{oo})}=105 n+12-30 j,\quad m_2^{(r_{oo})}=105 n-3-30 j,\\ 
&m_6^{(r_{oo})}=105 n-9-30 j,\quad m_1^{(r_{oo})}=105 n-24-30 j,\\ 
&m_5^{(r_{oo})}=105 n-30-30 j,\quad m_0^{(r_{oo})}=105 n-45-30 j\,. 
\end{align*}
\label{lem:tri456}
\end{Lem} 

\begin{Rem}
\label{rem:4-5-6}  
For simplicity, for example, write $g_n=g(t_4,t_5,t_6;n)$.  
We have $g_0=89$, $g_1=119$, $g_2=149$, $g_3=179$, $g_4=194$, $g_5=209$, $g_6=224$, $g_7=239$, $g_8=254$, $g_9=269$, $g_{10}=284$, $g_{11}=299$, $g_{13}=314$, $\dots$, $g_{49}=584$, $g_{51}=599$, $g_{54}=614$, $\dots$, $g_{201}=1154$, $g_{206}=1169$, $g_{212}=1184$.  
Note that the number of representations of $n$ in terms of $t_4,t_5,t_6$ cannot be $12,\dots,50,52,53,\dots,202,203,204,205,207,208,209,210,211,\dots$. Hence, for example, $g_{206}=g_{207}=g_{208}=g_{209}=g_{210}=g_{211}$. 

We have $n_0=45$, $n_1=74$, $n_2=104$, $n_3=134$, $n_4=149$, $n_5=164$, $n_6=179$, $n_7=194$, $n_8=209$, $n_9=224$, $n_{10}=239$, $n_{11}=254$, $n_{13}=269$, $\dots$, $n_{49}=539$, $n_{51}=554$, $n_{54}=569$, $\dots$, $n_{201}=1109$, $n_{206}=1124$, $n_{212}=1139$.  

We have $s_0=1455$, $s_1=3240$, $s_2=5925$, $s_3=9510$, $s_4=11640$, $s_5=13995$, $s_6=16575$, $s_7=19380$, $s_8=22410$, $s_9=25665$, $s_{10}=29145$, $s_{11}=32850$, $s_{13}=36780$, $\dots$, $s_{49}=145995$, $s_{51}=154200$, $s_{54}=162630$, $\dots$, $s_{201}=615960$, $s_{206}=632715$, $s_{212}=649695$.  
\end{Rem} 

\begin{proof}[Proof of Lemma \ref{lem:tri456}.] 
The sequence of the generating function 
$$
\frac{1}{(1-z^5)(1-z^7)}=\sum_{n=0}^\infty d_n z^n
$$ 
has the period $35$, and $g(5,7;p)=35 p+23$ ($p\ge 0$). 
Since $d(35 p+23;5,7)=p$, we have $d\bigl(3(35 p+23);15,21\bigr)=p$. 
As $d(n;15,21)=0$ for $n\not\equiv 0\pmod 3$, we see that 
\begin{align*} 
&d(3 n;10,15,21)=d(3 n+10;10,15,21)=d(3 n+20;10,15,21)\\
&=d(3 n;15,21)+d(3 n-30;15,21)+d(3 n-60;15,21)\\
&\qquad +\cdots+d\bigl(3(n-10\fl{n/10});15,21\bigr)\\
&=d(n;5,7)+d(n-10;5,7)+d(n-20;5,7)+\cdots+d\bigl(n-10\fl{n/10};5,7\bigr)\,. 
\end{align*} 
Hence, for even $p=2 p'$, we have 
\begin{align*} 
&d(105 p+69;10,15,21)=d(105 p+79;10,15,21)=d(105 p+89;10,15,21)\\
&=d(35 p+23;5,7)+d(35 p+13;5,7)+d(35 p+3;5,7)+\cdots\\
&\qquad +d(23;5,7)+d(13;5,7)+d(3;5,7)\\ 
&=\underbrace{p+\cdots+p}_4+\underbrace{(p-1)+\cdots+(p-1)}_3+\underbrace{(p-2)+\cdots+(p-2)}_4\\
&\qquad +\underbrace{(p-3)+\cdots+(p-3)}_3+\cdots
+0+0+0\\ 
&=4\bigl((2 p')+(2 p'-2)+\cdots+2\bigr)+3\bigl((2 p'-1)+(2 p'-3)+\cdots+1\bigr)\\
&=p'(7 p'+4)=\frac{p(7 p+8)}{4}\,. 
\end{align*} 
For odd $p=2 p'+1$, we have 
\begin{align*} 
&d(105 p+69;10,15,21)=d(105 p+79;10,15,21)=d(105 p+89;10,15,21)\\
&=d(35 p+23;5,7)+d(35 p+13;5,7)+d(35 p+3;5,7)+\cdots\\
&\qquad +d(18;5,7)+d(8;5,7)\\  
&=4\bigl((2 p'+1)+(2 p'-1)+\cdots+1\bigr)+3\bigl((2 p')+(2 p'-2)+\cdots+2\bigr)+0+0\\
&=(p'+1)(7 p'+4)=\frac{(p+1)(7 p+1)}{4}\,. 
\end{align*} 
Hence, for even $p$, 
$$
g(10,15,21;p(7 p+8)/4)=105 p+89\quad\hbox{and}\quad m_9^{(p(7 p+8)/4)}=105 p+99\,. 
$$ 
For odd $p$, 
$$
g(10,15,21;(p+1)(7 p+1)/4)=105 p+89\quad\hbox{and}\quad m_4^{((p+1)(7 p+1)/4)}=105 p+99\,. 
$$ 
Similarly, for even $p$, $d(105 p+39;10,15,21)=d(105 p+49;10,15,21)=d(105 p+59;10,15,21)$, $d(105 p+9;10,15,21)=d(105 p+19;10,15,21)=d(105 p+29;10,15,21)$, $d(105 p-21;10,15,21)=d(105 p-11;10,15,21)=d(105 p-1;10,15,21)$ and $d(105 p-51;10,15,21)=d(105 p-41;10,15,21)=d(105 p-31;10,15,21)$ are given by 
\begin{align*} 
\frac{p(7 p+8)}{4}-p&=\frac{p(7 p+4)}{4}\,,&~ 
\frac{p(7 p+8)}{4}-2 p&=\frac{7 p^2}{4}\,,\\ 
\frac{p(7 p+8)}{4}-3 p&=\frac{p(7 p-4)}{4}\,,&~ 
\frac{p(7 p+8)}{4}-4 p&=\frac{p(7 p-8)}{4}\,, 
\end{align*} 
respectively.   
For odd $p$, $d(105 p+39;10,15,21)=d(105 p+49;10,15,21)=d(105 p+59;10,15,21)$, $d(105 p+9;10,15,21)=d(105 p+19;10,15,21)=d(105 p+29;10,15,21)$, $d(105 p-21;10,15,21)=d(105 p-11;10,15,21)=d(105 p-1;10,15,21)$ and $d(105 p-51;10,15,21)=d(105 p-41;10,15,21)=d(105 p-31;10,15,21)$ are given by 
\begin{align*} 
\frac{(p+1)(7 p+1)}{4}-p&=\frac{7 p^2+4 p+1}{4}\,,\\ 
\frac{(p+1)(7 p+1)}{4}-2 p&=\frac{7 p^2+1}{4}\,,\\ 
\frac{(p+1)(7 p+1)}{4}-3 p&=\frac{7 p^2-4 p+1}{4}\,,\\ 
\frac{(p+1)(7 p+1)}{4}-4 p&=\frac{(p-1)(7 p-1)}{4}\,,  
\end{align*}  
respectively.   

Next, for even $p$, we have 
\begin{align*} 
&d(105 p+24;10,15,21)=d(105 p+34;10,15,21)=d(105 p+44;10,15,21)\\
&=d(35 p+8;5,7)+d(35 p-2;5,7)+d(35 p-12;5,7)+\cdots\\
&\qquad +d(24;5,7)+d(14;5,7)+d(4;5,7)\\ 
&=p+p+\underbrace{(p-1)+\cdots+(p-1)}_4+\underbrace{(p-2)+\cdots+(p-2)}_3\\
&\qquad +\cdots+\underbrace{2+\cdots+2}_3
+\underbrace{1+\cdots+1}_4\\ 
&=\frac{p(7 p+2)}{4}\,. 
\end{align*} 
For odd $p$, we have 
\begin{align*} 
&d(105 p+24;10,15,21)=d(105 p+34;10,15,21)=d(105 p+44;10,15,21)\\
&=p+p+\underbrace{(p-1)+\cdots+(p-1)}_4+\underbrace{(p-2)+\cdots+(p-2)}_3\\
&\qquad +\cdots+\underbrace{2+\cdots+2}_4
+\underbrace{1+\cdots+1}_3\\ 
&=\frac{7 p^2+2 p-1}{4}\,. 
\end{align*} 
Hence, for even $p$, 
$$
g(10,15,21;p(7 p+2)/4)=105 p+44\quad\hbox{and}\quad m_4^{(p(7 p+2)/4)}=105 p+54\,. 
$$ 
For odd $p$, 
$$
g(10,15,21;(7 p^2+2 p-1)/4)=105 p+44\quad\hbox{and}\quad m_9^{((7 p^2+2 p-1)/4)}=105 p+54\,. 
$$ 
Similarly, for even $p$, $d(105 p+14;10,15,21)=d(105 p+4;10,15,21)=d(105 p-6;10,15,21)$ is given by 
\begin{align*} 
\frac{p(7 p+2)}{4}-p=\frac{p(7 p-2)}{4}\,. 
\end{align*}  
For odd $p$, $d(105 p+14;10,15,21)=d(105 p+4;10,15,21)=d(105 p-6;10,15,21)$ is given by 
\begin{align*} 
\frac{7 p^2+2 p-1}{4}-p=\frac{7 p^2-2 p-1}{4}\,. 
\end{align*} 
\end{proof}

\subsection{Larger $n$}  

We can continue to obtain explicit formulas of the functions $g$, $n$ and $s$ for $n=5,6,\dots$, but the situation becomes more complicated. For example, concerning $g_p:=g(t_5,t_6,t_7;p)$, we see that 
\begin{align*}  
&g_0=125,~g_1=209,~g_2=230,~g_3=293,~g_4=314,~g_5=335,\\
&g_6=377,~g_7=398,~g_8=419,~g_9=440,~g_{10}=461,~g_{11}=482,\\
&g_{12}=503,~g_{13}=524,~g_{14}=545,~g_{16}=566,~\dots\,. 
\end{align*} 
Note that $g_{15}=g_{14}$ because the (largest) integer whose number of representations is exactly $15$ does not exist.

\section{Conclusion}   

In this paper, we are successful to obtain certain kinds of natural generalizations of Frobenius numbers together with Sylvester numbers and sums for three variables $(a_1,a_2,a_3)$, when they are consecutive triangular numbers. In other words, up to the present, any explicit formula for the maximum number, total number, and sum of numbers that can be expressed at most in $p$ ways have not been found for three variables or more. 
In this paper, explicit formulas are given when the triangular number is small, or when $p$ is small. By continuing the method presented here, we will be able to obtain explicit formulas for larger triplets and larger $p$'s, but the discussion must become more complicated.

 



\end{document}